\documentclass[11pt,reqno]{amsart}
\usepackage{amsthm}
\usepackage{siunitx}
\usepackage{amsfonts}
\usepackage{amssymb}
\usepackage{amsmath}
\usepackage{verbatim}
\usepackage{color}
\newcommand {\C}  {{\mathbb C}}
\newcommand {\CC}  {{\mathbb C}}

\newcommand {\NN}  {{\mathbb N}}
\newcommand {\N}  {{\mathbb N}}

\newcommand {\HH}  {{\mathcal H}}
\newcommand {\g}   {\mathfrak g}

\DeclareMathOperator{\Gal}{Gal}

\DeclareMathOperator{\Mon}{Mon}

\DeclareMathOperator{\Aut}{Aut}

\newcommand{\abs}[1]{\lvert#1\rvert}

\newtheorem {Def}{Definition}
\newtheorem {Th}{Theorem}

\newtheorem {Pro}{Proposition}
\newtheorem {Le}{Lemma}

\begin{document}
\title[Decomposable polynomials in recurrence sequences]{Decomposable polynomials in second order linear recurrence sequences}
\author{Clemens Fuchs}
\address[Clemens Fuchs]{University of Salzburg, Hellbrunnerstr. 34/I, 5020 Salzburg, Austria}
\email{clemens.fuchs@sbg.ac.at}
\author{Christina Karolus}
\address[Christina Karolus]{University of Salzburg, Hellbrunnerstr. 34/I, 5020 Salzburg, Austria}
\email{christina.karolus@sbg.ac.at}
\author{Dijana Kreso}
\address[Dijana Kreso]{Graz University of Technology, Steyrergasse 30/III, 8010 Graz, Austria}
\email{kreso@math.tugraz.at}


\begin{abstract}
We study elements of second order linear recurrence sequences $(G_n)_{n= 0}^{\infty}$ of polynomials in $\CC[x]$ which are decomposable, i.e.\@ representable as $G_n=g\circ h$ for some $g, h\in \C[x]$ satisfying $\deg g,\deg h>1$. Under certain assumptions, and provided that $h$ is not of particular type, we show that $\deg g$ may be bounded by a constant independent of $n$, depending only on the sequence.
\end{abstract}
\maketitle

\section{Introduction and Results}

Let $d\ge2$ be an integer. We consider a sequence of polynomials $(G_n)_{n=0}^{\infty}$  in $\C[x]$ satisfying the $d$-th order linear recurrence relation
\begin{equation}\label{relation}
G_{n+d}(x)=A_{d-1}(x)G_{n+d-1}(x)+\cdots+A_0(x)G_n(x), \quad n\in\N,
\end{equation}
\noindent determined by $A_0,A_1,\ldots,A_{d-1}\in \C[x]$ and initial terms $G_0,G_1,\ldots,G_{d-1}\in \C[x]$.
Let $\mathcal{G}\in \C(x)[T]$ be the characteristic polynomial of the sequence and let $\alpha_1,\ldots,\alpha_t$ be its distinct roots in the splitting field $L/\C(x)$ of $\mathcal{G}$, that is
\[
\mathcal{G}(T)=T^d-A_{d-1}T^{d-1}-\cdots - A_0=(T-\alpha_1)^{k_1}(T-\alpha_2)^{k_2}\cdots(T-\alpha_t)^{k_t},
\]
where $k_1,\ldots,k_t\in\NN$.
Then $G_n$ admits a representation of the form
\begin{align}
G_n(x)=\pi_1\alpha_1^n+\pi_2\alpha_2^n+\cdots+\pi_d\alpha_d^n,\label{binet}
\end{align}
where $\pi_i\in L[n]$ for $i=1, 2, \ldots, n$.
We say that the recurrence relation \eqref{relation} is \emph{minimal} if $(G_n)_{n=0}^{\infty}$ does not satisfy a recurrence relation with smaller $d$ and coefficients in $\C[x]$. We say that \eqref{relation} is \emph{non-degenerate} if $\alpha_i/\alpha_j\not\in \C^*$ for all $i\neq j$. Finally, we  say that \eqref{relation} is  \emph{simple} if $k_1=\cdots=k_t=1$; in this case the $\pi_i$'s lie in $L$. We also call the corresponding sequence $(G_n)_{n=0}^{\infty}$  minimal, non-degenerate and simple, respectively. In this paper, we will be concerned with second-order minimal non-degenerate simple linear recurrences.

Many Diophantine problems involving  linear recurrence sequences have been studied in the literature. For example, a famous problem is to estimate the number of zeros appearing in such a sequence, and more generally, to bound the number of solutions $n\in\N$ of the equation $G_n(x)=a$, where $a\in L$ is given (cf. \cite{fp} and the papers cited therein). Also, several authors studied the problem of giving bounds on $m$ and $n$ such that $G_n(x)=cG_m(P(x))$, $c=c(n, m)$, where $(G_n)_{n=0}^\infty$ is a linear recurrence sequence and $P$ a fixed polynomial (cf. \cite{f04,fpt02,fpt,fpt08}).

In this paper, we focus on \emph{decomposable} polynomials in second order linear recurrence sequences. A polynomial $f\in \C[x]$ with $\deg f>1$ is said to be decomposable
if it can be written as the composition $f(x)=g(h(x))$ with $g,h\in \C[x]$ and $\deg g, \deg h>1$, and \emph{indecomposable} otherwise.
The possible ways of writing a polynomial as a composition of polynomials were studied by several authors, starting with Ritt in the 1920's in his classical paper~\cite{R22}.
Results in this area of mathematics have applications to various other fields, e.g.\@ number theory, complex analysis, arithmetic dynamics, finite geometries, etc.\@  For example, there are applications to Diophantine equations of type $f(x)=g(y)$. In 2000, Bilu and Tichy~\cite{BT00},  by building on the work of  Siegel, Ritt, Fried and Schinzel,
classified the polynomials $f, g$ for which the equation $f(x)=g(y)$ has infinitely many solutions  in $S$-integers $x, y$. It turns out that such $f$ and $g$ must be representable as a composition of polynomials in a certain prescribed way.

In this paper we show that  if $(G_n)_{n=0}^{\infty}$ satisfies \eqref{relation} with $d=2$, under certain  assumptions on $G_0, G_1, A_0$ and $A_1$, if $G_n(x)=g(h(x))$ and $h(x)$ is not of particular type, then $\deg g$ may be bounded by a constant independent of $n$, depending only on the sequence (more precisely, it depends only on the degrees of $G_0, G_1, A_0, A_1$).
To describe what we mean by $h$ being of particular type and to state our results, we introduce the following notions. We say that $f,g\in \C[x]$ are {\it equivalent} if there are linear $\ell_1, \ell_2\in \C[x]$ such that $f(x)=\ell_1(x)\circ g(x)\circ \ell_2(x)$. For $f \in \C[x]$, we say that $f$ is {\it cyclic} if it is equivalent to a polynomial $g$ with $g(x)=x^n$ for some $n>1$, and we say that $f$ is {\it dihedral} if it is equivalent to $T_n$ for some $n>2$, where $T_n$ is a Chebychev polynomial, defined by the functional equation $T_n(x+1/x)=x^n+1/x^n$. Cyclic and dihedral polynomials play an important role in Ritt's theory of polynomial decomposition, as will be explained in Section \ref{sec2}.

To see that at least some exceptional cases have to be taken into account, consider e.g.\@ the well-known family of Fibonacci polynomials $F_n$, defined by
\begin{equation}\label{fibo}
F_0(x)=0, \quad F_1(x)=1, \quad F_{n+2}(x)=x F_{n+1}(x)+F_{n}(x)\mbox{ for } n\in\N.
\end{equation}
It is easy to see that for all odd $n\ge 3$, $F_n$ is an even polynomial of degree $n-1$, and hence if $n\ge5$ is odd, $F_n(x)$ can be written as $F_n(x)=g(h(x))$,
where $h(x)=x^2$ and $\deg g=(n-1)/2$. Clearly, here the degree of $g$ cannot be bounded independently of $n$. In this case, $h$ is cyclic.

Also, for Chebyshev polynomials $T_n$, which satisfy the second order linear recurrence
\[
T_0(x)=1,\quad T_1(x)=x, \quad T_{n+2}(x)=2xT_{n+1}(x)-T_{n}(x)\mbox{ for } n\in\N,
\]
it is well-known that $T_{mn}=T_m\circ T_n$ for any $m, n\in \N$. Since $\deg T_n=n$, clearly one cannot bound $\deg g$ independently of $n$ assuming $T_n(x)=g(h(x))$ and $\deg h>1$. In this case, $h$ is dihedral.

There is a third, trivial situation where it is clearly not possible to bound the degree of $g$ independently of $n$ assuming $G_n=g\circ h$, namely when $G_m(x)\in\C[h(x)]$ for every $m\in\N$. Consider for example the sequence $(F_n(h(x)))_{n=0}^\infty$, where $F_n$ is defined by \eqref{fibo} and $h\in \C[x]$. This sequence satisfies a second order linear recurrence relation and we clearly cannot bound $\deg F_n$ independently of $n$. It will be shown later that $G_m(x)\in\C[h(x)]$ for all $m\in\N$ if and only if $G_0,G_1,A_0,A_1\in\C[h(x)]$, see Lemma \ref{newl}.

We now describe our strategy and results in detail. Let $(G_n)_{n=0}^\infty$ be a minimal non-degenerate simple second order linear recurrence sequence given by \eqref{relation} (with $d=2$). Assume that $G_n$ is decomposable for some $n\in \N$  and write $G_n(x)=g(h(x))$, where $h$ is indecomposable, and thus $\deg h\geq 2$. By Gauss's lemma it follows that the polynomial $h(X)-h(x)\in\C(h(x))[X]$ is irreducible  and since $h'(X)\neq 0$, it is also separable (find details in Section~\ref{sec2}). Since $\deg h\geq 2$, there exists a root $y\neq x$ in its splitting field over $\C(h(x))$. Clearly, $h(x)=h(y)$.
As in \eqref{binet}, we have
\[
G_n(x)=\pi_1\alpha_1^n+\pi_2\alpha_2^n,
\]
where $\alpha_1, \alpha_2$  are distinct roots  of the characteristic polynomial $\mathcal{G}_1(T)=T^2-A_{1}(x)T-A_0(x)$ in its splitting field $L_1/\C(x)$, and  $\pi_1, \pi_2\in L_1$. Indeed, there is a representation of this form since by assumption the characteristic polynomial has no multiple roots. Observe that $\pi_i\alpha_i^n\neq0$ for all $n\in \N$ and $i=1,2$ by minimality.
Conjugating (in some fixed algebraic closure of $\C(x)$ containing $\alpha_1,\alpha_2$) over $\C(h(x))$ via $x\mapsto y$, we get a sequence $(G_n(y))_{n=0}^\infty$ with $G_n(y)\in\C[y]$, which satisfies the same minimal non-degenerate simple recurrence relation as $(G_n(x))_{n=0}^\infty$ with $x$ replaced by $y$. We conclude that
\[
G_n(y)=\rho_1\beta_1^n+\rho_2\beta_2^n,
\]
where $\beta_1, \beta_2$ are distinct roots  of the characteristic polynomial $\mathcal{G}_2(T)=T^2-A_{1}(y)T^{d-1}-A_0(y)$ in its splitting field $L_2/\C(y)$, and $\rho_1,\rho_2\in L_2$. Again we have that $\rho_i\beta_i^n\neq 0$ for all $n\in\N$ and $i=1,2$. Since $h(x)=h(y)$, we get $G_n(x)=G_n(y)$, that is
\begin{equation}\label{sum}
\pi_1\alpha_1^n + \pi_2\alpha_2^n=\rho_1\beta_1^n+\rho_2\beta_2^n.
\end{equation}
We view this last equation as an $S$-unit equation in function fields and seek to apply a result of Brownawell and Masser (see Theorem \ref{bm} below) to bound the height of $G_n$ and consequently the degree of $g$. However, this theorem can be applied directly only to equations in which no proper subsum vanishes.
We will show in Section~\ref{sec5} that if $h$ is not cyclic, then equation \eqref{sum} has a proper vanishing subsum if and only if
\[
\pi_1\pi_2A_0(x)^n\in \C(h(x)).
\]
In particular, the existence of a proper vanishing subsum of \eqref{sum} does not depend on the choice of the conjugate $y$ of $x$ over $\C(h(x))$. Note that if $h$ is not cyclic and $A_0(x)=a_0\in \C$, $\pi_1\pi_2=\pi\in \C$,  then there  exists a vanishing subsum of \eqref{sum} and one cannot apply the theorem in question; for example, this is the case for Chebyshev polynomials $T_n$.

We now state our main result.

\begin{Th}\label{thm1}
Let $A_0, A_1, G_0, G_1\in \C[x]$ and $(G_n)_{n=0}^\infty$ be a sequence of polynomials defined by the minimal non-degenerate simple linear recurrence
\begin{equation}\label{equat}
G_{n+2}(x)=A_{1}(x)G_{n+1}(x)+A_0(x)G_n(x), \quad n\in\N.
\end{equation}
There is a positive real constant $C=C(\{A_i, G_i : i=1,2\})$ with the following property. If for some $n$ we have $G_n(x)=g(h(x))$, where $h$ is indecomposable and neither dihedral nor cyclic, and if \eqref{sum} has no proper vanishing subsum, then it holds that $\deg g\leq C$.
\end{Th}

We mention that the constant $C$ in Theorem~\ref{thm1} can be effectively computed; this is done in the proof of  the theorem. Since the bound is not very illuminating, we have not stated it above. Also note that in the theorem the situation that $G_m(x)\in \C[h(x)]$ for all $m$ is not excluded explicitly. It will be shown (see Lemma~\ref{third}) that in this case either $h$ is cyclic or equation \eqref{sum} has a proper vanishing subsum.

Theorem~\ref{thm1} resembles a result of Zannier~\cite{Z07}, who showed that if $f$ is a polynomial with $\ell$ non-constant terms and $f(x)=g(h(x))$, where $h$ is not of type $ax^k+b$, $a\neq 0$, then $\deg g\leq 2\ell(\ell-1)$. Our proof, like Zannier's proof,  involves applying Brownawell and Masser's theorem~\cite{BM86}. The application of this theorem in our proof requires a different approach and the technical details are more challenging. We remark that Zannier's result was one of the main ingredients of the proof of a conjecture of Schinzel ~\cite{Z08} by the same author, which states that for $f\in \C[x]$ with $\ell$ non-constant terms, satisfying $f=g\circ h$ for some  $g, h\in \C[x]$, the number of terms of $h$ is bounded above by $B(\ell)$, where $B$ is an explicitly computable function. Zannier's result was then used in \cite{K15+, K15} to study Diophantine equations of type $f(x)=g(y)$, where $f$ and $g$ are arbitrary polynomials with a fixed number of non-constant terms, via the criterion of Bilu and Tichy. We remark that likewise, using our results, one may study Diophantine equations of this type where $f$ and/or $g$ are elements of a second order linear recurrence sequence of polynomials. We further mention that some special cases of the latter problem have already been studied in the literature, see \cite{DT01, kp}.

To detect cases when there does not exist a vanishing subsum of \eqref{sum}, we apply several tools. We follow a Galois-theoretic approach to decomposition questions, which originated in Ritt's work~\cite{R22}, and apply some recent results on polynomial decomposition from \cite{BWZ09} and \cite{ZM}. We show that the following holds.

\begin{Th}\label{main}
Let $A_0, A_1, G_0, G_1\in \C[x]$ and $(G_n)_{n=0}^\infty$ be a sequence of polynomials defined by the minimal non-degenerate simple linear recurrence
$$G_{n+2}(x)=A_{1}(x)G_{n+1}(x)+A_0(x)G_n(x), \quad n\in\N.$$
Assume that for some $n$ we have $G_n(x)=g(h(x))$, where $h$ is indecomposable.
If $h$ is  neither dihedral nor cyclic, and it does not hold that $G_m(x)\in\C[h(x)]$ for all $m\in\N$, then \eqref{sum} has no proper vanishing subsum if $A_0(x)$ is constant and any of the following holds:
\begin{enumerate}
\item[i)]  $2G_1(x)=G_0(x)A_1(x)$,  \textnormal{i.e}.\@ $\pi_1=\pi_2$,
\item[ii)]  $G_1(x)=2A_0(x)+G_0(x)^2$, $G_0(x)=A_1(x)$,
\item[iii)]  $G_1(x)=-2A_0(x)$, $G_0(x)=A_1(x)$.
\end{enumerate}
If $h$ is not cyclic, and it does not hold that $G_m(x)\in\C[h(x)]$ for all $m\in\N$, then \eqref{sum} has no proper vanishing subsum if any of the following holds:
\begin{enumerate}
\item[i)]  $\pi_1\pi_2=\pi A_0(x)^m$, for some $\pi \in \C$, $m\geq 0$ and $\deg A_0=1$,
\item[ii)]  $\pi_1=\pi_2=\pi\in\C$ and either $\sqrt{A_1(x)^2+4A_0(x)}\in \C[x]$ or \\$\deg A_0=1$,
\item[iii)]  $G_1(x)=2A_0(x)+G_0(x)^2$, $G_0(x)=A_1(x)$ and  \\either $\sqrt{A_1(x)^2+4A_0(x)}\in \C[x]$ or $\deg A_0=1$,
\item[iv)]   $G_1(x)=-2A_0(x)$, $G_0(x)=A_1(x)$ and \\either $\sqrt{A_1(x)^2+4A_0(x)}\in\C[x]$ or $\deg A_0=1$.
\end{enumerate}
\end{Th}

We mention that the condition $\sqrt{A_1(x)^2+4A_0(x)}\in \C[x]$ means that the roots $\alpha_1,\alpha_2$  of the corresponding characteristic  polynomial are in $\C[x]$. As clarified in the theorem, the condition $2G_1(x)=G_0(x)A_1(x)$ is equivalent to the condition $\pi_1=\pi_2$. Furthermore, we mention that if $G_0(x)=A_1(x)$, and either $G_1(x)=2A_0(x)+G_0(x)^2$ or $G_1(x)=-2A_0(x)$, then either $\pi_1=\alpha_1$ and  $\pi_2=\alpha_2$, or  $\pi_1=\alpha_2$ and  $\pi_2=\alpha_1$ (see Lemma~\ref{lem1} and Lemma~\ref{lem2} for more details).

The paper is organized as follows. In Section \ref{sec2} we shall collect some facts about polynomial decomposition; here Galois-theoretic arguments play an important role. In Section \ref{sec3} we collect auxiliary results concerning heights in function fields, state some well-known theorems from the literature, and prove three lemmas which will be used to prove our main results. In Section \ref{sec5} we give a proof of Theorem \ref{main} using results from the previous two sections. In Section \ref{sec4} we give a proof of Theorem \ref{thm1}. As already mentioned above, our proof of Theorem \ref{thm1} involves applying the theory of $S$-unit equations over function fields.

\section{Polynomial decomposition via Galois theory}\label{sec2}

Recall that a polynomial $f\in \C[x]$ with $\deg f>1$ is called \emph{indecomposable}
if it cannot be written as the composition $f(x)=g(h(x))$ with $g,h\in \C[x]$, $\deg g>1$ and $\deg h>1$. Otherwise, $f$ is said to be \emph{decomposable}. Any representation of $f$ as a functional composition of  polynomials of degree $>1$ is said to be a \emph{decomposition} of $f$. A decomposition $f=f_1\circ f_2\circ \cdots\circ f_m$ of $f$ is said to be \emph{complete} if each $f_i$ is an indecomposable polynomial.

Note that if $\mu \in \C[x]$ is linear, then there exists $\mu^{\langle-1\rangle}\in \C[x]$ such that $(\mu \circ \mu^{\langle-1\rangle})(x)=(\mu^{\langle-1\rangle}\circ \mu)(x)=x$. Thus, $g\circ h=g\circ \mu \circ \mu^{\langle-1\rangle}\circ h$. By comparison of degrees one sees that no such polynomial exists when $\deg \mu>1$.

\begin{Def}\label{mon}
Given $f\in \C[X]$ with $\deg f>1$, the \emph{monodromy group} $\Mon(f)$ of $f$  is the Galois group of $f(X)-t$ over the field $\C(t)$, where $t$ is transcendental, viewed as a group of permutations of the roots of $f(X)-t$.
\end{Def}

A lot of information about the polynomial $f$ is encoded into its mono-dromy group. By Gauss's lemma it follows that $f(X)-t$  is irreducible over $\C(t)$, so $\Mon(f)$ is a transitive permutation group. Since $f'(X)\neq 0$, it follows that $f(X)-t$ is also separable. Let $x$ be a root of $f(X)-t$ in its splitting field $L$ over $\C(t)$. Then $t=f(x)$ and $\Mon(f)=\Gal(L/\C(f(x)))$ is viewed as a permutation group on the conjugates of $x$ over $\C(f(x))$.

L\"uroth's theorem (see \cite[p.\@~13]{S00}) states that for a field $K$ satisfying $\C\subset K\subseteq \C(x)$ we have $K=\C(h(x))$ for some $h\in \C(x)$. This theorem provides a dictionary between decompositions of $f\in \C[x]$ and fields between $\C(f(x))$ and $\C(x)$. Namely, if $f(x)=g(h(x))$, then $\C(f(x))\subseteq \C(h(x))\subseteq \C(x)$. On the other hand, if $K$ is a field between $\C(f(x))$ and $\C(x)$, by L\"uroth's theorem it follows that $K=\C(h(x))$ for some $h\in \C(x)$. Since $f$ is a polynomial, $h$ can be chosen to be a polynomial by \cite[p.~16]{S00}. Then $f=g(h(x))$ for some $g\in \C[x]$.
The  fields between $\C(f(x))$ and $\C(x)$ clearly correspond to groups between the two associated Galois groups -- $\Gal(L/\C(f(x)))=\Mon(f)=:G$ and $\Gal(L/\C(x))=:H$  (the stabilizer of $x$ in $\Mon(f)$). In this way, the study of ways to represent a polynomial $f$  as a composition of lower degree polynomials reduces to a study of subgroups of the monodromy group of $f$, and more precisely to the study of groups between $H$ and $G$. Furthermore, it can be shown that $G$ has a transitive cyclic subgroup, that is that $G=HI$ for some cyclic group $I$ ($I$ can be chosen to be the inertia group at any place of the splitting field of $f(x)-t$ which lies over the infinite place of $\C(t)$); see also  \cite[Lemma~3.4]{KZ14} or \cite[Lemma~3.3]{T95}. In this way, the study of ways to represent a complex polynomial $f$  as a composition of lower degree polynomials reduces to a study of subgroups of the cyclic group $I$.

The interested reader is referred to \cite{KZ14} and \cite{ZM} to find out more about the Galois-theoretic setup for addressing decomposition questions which originated in Ritt's work~\cite{R22}. Ritt~\cite{R22} showed that any complete decomposition of a complex polynomial $f$ can be obtained from any other through a sequence of steps, each of which involves replacing two adjacent indecomposables by two others with the same composition. He then solved the equation $a\circ b=c\circ d$ in indecomposable complex polynomials, showing that the only solutions, up to composing with linear polynomials, are the trivial one $a\circ b=a\circ b$ and the non-trivial solutions
\[
x^n\circ x^kh(x^n)=x^kh(x)^n\circ x^n \ \textnormal{ and }\
T_m(x)\circ T_n(x)=T_n(x)\circ T_m(x),
\]
where $h\in \C[x]$, $n, k, m\in \N$ and $T_n$ is the $n$-th Chebyshew polynomial defined in the introduction.
We now record two results on the topic that we will repeatedely use in the sequel.

\begin{Pro}\label{invar}
Pick $f\in \C[x]$ of degree $\deg f > 1$.  For any two complete decompositions $f=f_1 \circ f_{2} \circ \cdots \circ f_m=g_1\circ \ g_{2} \circ \cdots \circ g_{n}$ of $f$,
we have that $m=n$ and $\textnormal{Mon} (f_i)\cong\textnormal{Mon} ( g_{\sigma(i)})$
for some permutation $\sigma$ of the set $\{1, 2, \ldots, m\}$ and for all $i=1, 2, \ldots, m$.
\end{Pro}


\begin{Pro}\label{cydih}
Pick $f\in \C[x]$ of degree $n > 1$. Then $ \Mon(f)$ is cyclic if and only if $f$ is cyclic, in which case $\abs{\Mon(f)}=n$. Likewise, if $n > 2$, then $ \Mon(f)$ is dihedral if and only if $f$ is dihedral, in which case $\abs{\Mon(f)} = 2n$.
\end{Pro}

Recall that for $f \in \C[x]$, we say that $f$ is cyclic if it is equivalent to $x^n$ for some $n>1$, and we say  that $f$ is dihedral if it is equivalent to $T_n$ for some $n>2$. Proposition~\ref{invar}  is Theorem~1.3 in \cite{ZM}. See also \cite[Thm.~5.1]{KZ14}. Proposition~\ref{cydih} is Lemma~3.6 in \cite{ZM}. See also Theorem~3.8 in \cite{B99}. We record the following corollary.

\begin{Le}\label{sequen}
Pick $f\in \C[x]$ with $\deg f>1$. If $f$ is dihedral, then for any complete decomposition of $f$ the collection of  monodromy groups of the
indecomposable polynomials consists only of dihedral groups. Furthermore, if $f$ is cyclic, then for any complete decomposition of $f$ the collection of  monodromy groups of the
indecomposable polynomials consists only of cyclic groups.
\end{Le}
\begin{proof}
By Proposition \ref{cydih}, it suffices to prove the statement in the cases $f(x)=T_m(x)$ and $f(x)=x^m$ for $m\in\N$, respectively.
Note that since $T_{mn}(x)=T_m(T_n(x))$ for any $m, n\in \N$ and $\Mon(f)$ is dihedral if and only if $f$ is dihedral, for any $m\in \N$ there exists a complete decomposition of $T_m(x)$ such that the collection of  monodromy groups of the
indecomposable polynomials consists only of dihedral groups. By Proposition~\ref{invar}, for any complete decomposition of $T_m(x)$ the collection of  monodromy groups of the
indecomposable polynomials consists only of dihedral groups. By the same argument, for any complete decomposition of $x^m$ the collection of  monodromy groups of the
indecomposable polynomials consists only of cyclic groups.
\end{proof}

In the literature, quite often Ritt's and related results are expressed in terms of Dickson polynomials $D_n(x, a)$ (with parameter $a$), as they satisfy
\begin{equation}\label{dick}
D_n(2ax, a^2)=2a^nT_n(x),\ a\neq 0, \quad D_n(x, 0)=x^n.
\end{equation}
 We refer to Turnwald's paper~\cite{T95} for various properties of Chebyshev and Dickson polynomials. We now list some that will be of importance to us in this paper.

\begin{Pro}\label{propd}
All of the following holds:
\begin{itemize}
\item  $T_0(x)=1$, $T_1(x)=x$, $T_n(x)=2xT_{n-1}(x)-T_{n-2}(x)$, $n\geq 2$.
\item $D_0(x, a)=2$, $D_1(x, a)=x$, $D_n(x, a)=xD_{n-1}(x, a)-aD_{n-2}(x, a)$, $n\geq 2$.
\item $T_{mn}(x)=T_m(T_n(x))$ for any $m, n\in \N$.
\item $D_{mn}(x, a)=D_m(D_n(x, a), a^n)$ for any $m, n\in \N$.
\item $D_n(x, 0)=x^n$.
\item $D_n(x+a/x,a)=x^n+(a/x)^n$.
\item $D_n(x+y, xy)=x^n+y^n$.
\item Let $n\geq 2$ and let $\zeta_n\in \C$ be a primitive $n$-th root of unity. Put $\gamma_k=\zeta_n^k+\zeta_n^{-k}$ and $\delta_k=\zeta_n^k-\zeta_n^{-k}$ \textup{(}so that $\gamma_k^2-4=\delta_k^2$\textup{)}. Then
\end{itemize}
\[
D_n(x, a)-D_n(y,a)=(x-y) \prod_{k=1}^{(n-1)/2} (x^2-\gamma_kxy +y^2+\delta_k^2a),
\]
when $n$ is odd and
\[
D_n(x, a)-D_n(y,a)=(x-y)(x+y) \prod_{k=1}^{(n-2)/2} (x^2-\gamma_kxy +y^2+\delta_k^2a),
\]
when $n$ is even.
\end{Pro}

For the proof of Theorem \ref{main} we will also need the following result about polynomials with a common composite, which can be deduced from a result of Beals, Wetherell and Zieve~\cite[Thm.~5.1]{BWZ09}.
If $f_1, f_2\in \C[x]$ are non-constant polynomials for which there exist non-constant $u, v\in \C[x]$ such that $u(f_1(x))=v(f_2(x))$, then $f_1$ and $f_2$ are said to have a \emph{common composite}. `Most' pairs of complex polynomials have no common composite (this follows to the most part already from Ritt's results, see \cite{BWZ09} for the details). The following fact will be repeatedy used in our proof of Theorem~\ref{main}.

\begin{Pro}\label{indcor}
Suppose $f_1, f_2\in \C[x]$ satisfy $\deg f_1>1, \deg f_2>1$ and $f_2$ is indecomposable. Then $f_1$ and $f_2$ have a common composite  if and only if there are linear polynomials $\ell_1, \ell_2, \ell_3\in \C[x]$ such that one of the following holds:
\begin{itemize}
\item $f_1(x)=\ell_1(x)\circ x^rP(x^n)\circ \ell_3(x)$ and $f_2(x)=\ell_2(x)\circ x^n\circ \ell_3(x)$, where $r, n>0$, $P\in \C[x]$, $\gcd(\deg f_1, \deg f_2)=1$ and $n$ is prime.
\item $f_1(x)=\ell_1(x)\circ x^n\circ \ell_3(x)$ and $f_2(x)=\ell_2(x)\circ x^rP(x^n)\circ \ell_3(x)$, where $r, n>0$, $P\in \C[x]$, $\gcd(\deg f_1, \deg f_2)=1$ and $x^rP(x^n)$ is indecomposable, so in particular $\gcd(r, n)=1$.
\item $f_1(x)=\ell_1(x)\circ D_m(x, \alpha) \circ \ell_3(x)$,  $f_2(x)=\ell_2(x)\circ D_n(x, \alpha) \circ \ell_3(x)$, where $m, n>1$, $\alpha\in \C$, $\gcd(\deg f_1, \deg f_2)=1$ and $n$ is prime.
\item $f_1(x)\in \C[f_2(x)]$.
\end{itemize}
\end{Pro}


\section{Preliminaries and auxiliary results}\label{sec3}

Our strategy involves the use of height functions in function fields.  In what follows, let $L$ be a finite extension of the rational function field $\C(x)$.
For $a\in \C$ define the valuation $\nu_a$ as follows. For $q(x)\in \C(x)$ let $q(x)=(x-a)^{\nu_a(q)}A(x)/B(x)$, where $A, B\in \C[x]$ and $A(a)B(a)\neq0$. Furthermore,  denote by $\nu_{\infty}$ the (only) infinite valuation which is defined by   $\nu_{\infty}(Q):=\deg B-\deg A$ for $Q(x)=A(x)/B(x)$, where $A, B\in \C[x]$. These are all (normalized) discrete valuations on $\C(x)$. All of them can be extended in at most $\left[L:\C(x)\right]$ ways to a discrete valuation on $L$ and again in this way one obtains all discrete valuations on $L$. Furthermore, for $f\in L^*$ the sum formula $\sum\limits{\nu(f)}=0$ holds, where the sum is taken over all discrete valuations on $L$. We just mention that there are different equivalent descriptions of the notion of discrete valuations as e.g. places or the rational points on a(ny) nonsingular complete curve over $\C$ with function field $L$.

Now, define the {\it projective height} $\mathcal{H}$ of $u_1,\ldots,u_n\in L/\C(x)$, where $n\ge 2$ and not all $u_i$ zero, via
\begin{align}
\mathcal{H}(u_1,\ldots,u_n)=-\sum\limits_{\nu}\min(\nu(u_1),\ldots,\nu(u_n)).
\end{align}
Also, for a single element $f\in L^*$, we set
\begin{align}
\mathcal{H}(f)=-\sum\limits_{\nu}\min(0,\nu(f)).
\end{align}
In both cases the sum is taken over all discrete valuations $\nu$ on $L$. Note that $\nu(f)\neq0$ only for a finite number of valuations $\nu$ and that $\HH(f)=\sum_{\nu}\max (0,\nu(f))$ if $f\in L^*$, by the sum formula. For $f=0$, we define $\HH(f)=\infty$. We call $a$ a {\it zero} of $f$ if $\nu_a(f)>0$ and a {\it pole} of $f$ if $\nu_a(f)<0$. We state some basic properties of the projective height.

\begin{Le} \label{propert}
Denote as above by $\mathcal{H}$ the projective height on $L/\C(x)$. Then for $f,g\in L^*$ the following properties hold:
\begin{center}
\begin{enumerate}
\item $\HH(f)\ge 0$ and $\HH(f)=\HH(1/f)$,
\item $\HH(f)-\HH(g)\le\mathcal{H}(f+g)\le \HH(f)+\HH(g)$,\label{plus}
\item $\HH(f)-\HH(g)\le\HH(fg)\le \HH(f)+\HH(g)$,\label{mult}
\item $\HH(f^n)=|n|\cdot\HH(f)$,
\item $\HH(f)=0\Leftrightarrow f\in \C^*$,
\item $\HH(A(f))=\deg A\cdot \HH(f)$ for any $A\in \C[T]\backslash\left\{0\right\}$.\label{last}
\end{enumerate}
\end{center}
\end{Le}

\begin{proof}
$\HH(f)\ge 0$ clearly holds by definition.
To show that $\mathcal{H}(f+g)\le \HH(f)+\HH(g)$, note that $\min(0,\nu(f+g))\ge\min(0,\nu(f))+\min(0,\nu(g))$. Namely, if $\min(0,\nu(f+g))=0$, this clearly holds. Otherwise, by the definition of discrete valuations we have $\nu(f+g)\ge\min(\nu(f),\nu(g))$ and it follows that $\min(0,\nu(f+g))=\nu(f+g)\ge\min(0,\nu(f))+\min(0,\nu(g))$. Hence, $\HH(f+g)=-\sum_{\nu}\min(0,\nu(f+g))\le-\sum_{\nu}\min(0,\nu(f))-\sum_{\nu}\min(0,\nu(g))=\HH(f)+\HH(g)$.
Similarly, $\HH(fg)\le\HH(f)+\HH(g)$ follows from $\nu(fg)=\nu(f)+\nu(g)$.

We now show that $\HH(f)=\HH(1/f)$. Since $f\neq0$, clearly $\nu(f^{-1})=-\nu(f)$ and therefore we have $\min(0,\nu(f))=-\max(0,\nu(f^{-1}))$. By the sum formula it follows that $\HH(f)=-\sum_\nu{\min(0,\nu(f))}=\sum_\nu{\max(0,\nu(f^{-1}))}=-\sum_\nu{\min(0,\nu(f^{-1}))}=\HH(f^{-1})$.

Next we show that $\HH(f)-\HH(g)\le\mathcal{H}(fg)$. We have $\HH(f)=\HH(fgg^{-1})\le\HH(fg)+\HH(g^{-1})=\HH(fg)+\HH(g)$, so $\HH(fg)\ge \HH(f)-\HH(g)$. Analogously, one concludes $\HH(f+g)\ge\HH(f)-\HH(g)$.

For $n\in\NN_0$, the identity $\HH(f^n)=|n|\cdot\HH(f)$ follows immediately from the definition of discrete valuations. Since $\HH(f^n)=\HH(f^{-n})$, the statement also holds for negative integers $n$.

By \cite[Cor.~I.1.19, p.~8]{st}, any transcendental element $f\in L$ has at least one zero and one pole. So if $f$ is transcendental, there is a valuation $\nu$ on $L$ such that $\nu(f)<0$ and consequently $\HH(f)>0$. On the other hand, $\HH(f)=0$ for any $f\in \C^*$.

To see that ~\eqref{last} holds, observe that
by (\ref{plus}) and (\ref{mult}), it follows that if $a\in \C$, then $\HH(af)=\HH(f+a)=\HH(f)$. We argue by induction on $n=\deg A$. The statement holds for $n=0$ since in this case $\HH(A(f))=0=\deg A\cdot\HH(f)$. Also, if $n=1$, and say $A(T)=aT+b$ where $a,b\in \C$, then $\HH(A(f))=\HH(af+b)=\HH(f)=\deg A\cdot\HH(f)$.
Let us now assume that $\deg A=n+1$ and that the statement is true for lower-degree polynomials.
If $A(T)=aT^{n+1}+b$, with $a,b\in \C$, the claimed equality clearly holds. Otherwise, let $m>0$ be the unique integer such that $A(T)-A(0)=T^mA_1(T)$ and $A_1(T)\in \C[T]$ is such that $A_1(0)\neq 0$. Note that $\deg A_1=n+1-m$, so that we can apply the induction hypothesis to $A_1$. We claim that
\begin{align*}
\max(0,\nu(f^m)+\nu(A_1(f)))=\max(0,\nu(f^m))+\max(0,\nu(A_1(f))).
\end{align*}
Indeed, if $\nu(f^m)>0$ then $\nu(f)>0$, and by the strict triangle inequality for valuations it follows that $\nu(A_1(f))=0$ for $A_1(0)\neq 0$. On the other hand, if $\nu(f^m)<0$, and consequently $\nu(f)<0$, then (again by the strict triangle inequality) we have $\nu(A_1(f))<0$. So the claimed equality holds in any case. We conclude
\begin{align*}
\HH(A(f))&=\HH(A(f)-A(0))=\HH(f^mA_1(f))=\sum_\nu\max(0,\nu(f^mA_1(f)))\\
&=\sum_\nu\max(0,\nu(f^m)+\nu(A_1(f)))\\
&=\sum_\nu\left[\max(0,\nu(f^m))+\max(0,\nu(A_1(f)))\right]\\
&=\HH(f^m)+\HH(A_1(f))=m\cdot\HH(f)+(n+1-m)\cdot\HH(f)\\
&=\deg A\cdot\HH(f).
\end{align*}
\end{proof}
We use the following result due to Brownawell and Masser taken from \cite{fz} (more precisely, this is a direct consequence of \cite[Thm.~B and Cor.~1]{BM86}), which gives an upper bound for the height of $S$-units, which arise as a solution of certain $S$-unit-equations. Recall that for a set $S$ of discrete valuations, we call an element of $L$ an $S${\it-unit}, if it has poles and zeros only at places in $S$, or equivalently, the set of $S$-units in $L$ is
\[
\mathcal{O}_S^*=\{f\in L: \nu(f)=0\mbox{ for all }\nu\notin S\}.
\]

\begin{Th}[Brownawell-Masser] \label{bm}
Let  $F/\C$ be a function field of one variable of genus $\g$. Moreover, let  $u_1,\ldots, u_n$ be not all constant $S$-units for a finite set $S$ of discrete valuations, and
$$1+u_1+u_2+\ldots +u_n=0,$$
where no proper subsum of the left side vanishes. Then it holds
\begin{align}
\max\limits_{i=1,\ldots,n}\HH(u_i)\le\frac{1}{2}(n-1)(n-2)(|S|+2\g-2).
\end{align}
\end{Th}

Furthermore, we use the following classical estimates for the genus of a compositum of function fields, which are taken from \cite[p.~130, p.~132]{st}.
\begin{Th}[Castelnuovo's Inequality]\label{castelnuovo}
Let $F/\C$ be a function field of one variable of genus $\g$. Suppose there are given two subfields $F_1/\C$ and $F_2/\C$ of $F/\C$ satisfying
\begin{enumerate}
\item $F=F_1F_2$ is the compositum of $F_1$ and $F_2$.
\item $\left[F:F_i\right]=n_i$, and $F_i/\C$ has genus $\g_i$ \textup{(}$i=1,2$\textup{)}.
\end{enumerate}
Then we have
\[
\g\le n_1\g_1+n_2\g_2+(n_1-1)(n_2-1).
\]
\end{Th}

\begin{Th}[Riemann's Inequality]\label{riemann}
Suppose that $F=\C(x,y)$. Then we have the following estimate for the genus $\g$ of $F/\C$:
\[
\g\le([F:\C(x)]-1)([F:\C(y)]-1).
\]
\end{Th}

We now prove three lemmas that we will need in the proofs of our main results.

\begin{Le}\label{extdeg}
Let $h\in \C[x]$ be indecomposable  and let  $y\neq x$ be a root of $h(X)-h(x)\in\C(x)[X]$. If $h$ is neither cyclic nor dihedral, then \[[\C(x,y):\C(x)]\ge\frac{1}{2}\deg h.\]
\end{Le}
\begin{proof}
We set $d=[\C(x,y):\C(x)]$. Then $d$ is the degree of a minimal polynomial $\tilde{H}(Y)\in \C(x)[Y]$ of $y$ over $\C(x)$. Let $H(X,Y)=(h(X)-h(Y))/(X-Y)\in \C[X,Y]$. Then $H(x,Y)\in \C(x)[Y]$ is a polynomial in $Y$ for which $H(x,y)=0$ holds. It follows that $\tilde{H}(Y)$ divides $H(x,Y)$.

If $H_1(X,Y)\in \C[X,Y]$ is any irreducible polynomial such that $H_1(x,y)=0$, then $H_1(X,Y)|H(X,Y)$. Then the highest homogeneous part of $H_1(X,Y)$ divides the highest homogeneous part of $H(X,Y)$, which is a constant multiple of
\[
\frac{X^{\deg h}-Y^{\deg h}}{X-Y}=X^{\deg h-1}+X^{\deg h-2}Y+\cdots+XY^{\deg h-2}+Y^{\deg h-1}.
\]
Therefore, it follows $\deg H_1=\deg_X H_1=\deg_YH_1=d$. This argument can be found in the proof of \cite[Lemma\@~3]{Z07}.

Since $h$ is neither cyclic nor dihedral, if $\deg h\ge 3$, according to Fried~\cite{F70}  it follows that $H(X,Y)=(h(X)-h(Y))/(X-Y)\in \C[X,Y]$ is irreducible. (See also Turnwald's paper~\cite[Thm.~4.5]{T95} for a detailed exposition of Fried's proof.) Then $H$ is a constant multiple of $H_1$ and we conclude
\[
\deg h-1=\deg H=\deg H_1=\deg_{Y}H_1=d.
\]
Thus, $[\C(x,y):\C(x)]=\deg h-1\ge\deg h/2$. If $\deg h=2$, we clearly have $[\C(x,y):\C(x)]\ge 1=\deg h/2$.
\end{proof}

\begin{Le}\label{either}
Let $h\in \C[x]$ be indecomposable and let  $y\neq x$ be a root of $h(X)-h(x)\in\C(x)[X]$. Then either $\C(x)\cap\C(y)=\C(x)$ and $h$ is cyclic or $\C(x)\cap\C(y)=\C(h(x))$.
\end{Le}
\begin{proof}
By assumption, $h(x)=h(y)$.
Note that thus $\C(h(x))\subseteq \C(x)\cap \C(y)\subseteq \C(x)$.
By L\"uroth's theorem (see \cite[p.\@~13]{S00}) it follows that $\C(x)\cap \C(y)=\C(r(x))$ for some $r\in \C(x)$. Moreover, since $h$ is a polynomial, $r$ can be chosen to be a polynomial as well by \cite[p.~16]{S00}. Assume henceforth $r\in \C[x]$. Then $h(x)\in \C[r(x)]$. Since $h$ is indecomposable, it follows that either $\deg r=\deg h$ or $\deg r=1$, i.e.\@ that either $\C(x)\cap \C(y)=\C(h(x))$ or $\C(x)\cap \C(y)=\C(x)$. Note that if $\C(x)\cap \C(y)=\C(x)$, then $\nu(y)=x$ for some $\nu\in \C(x)$. Furthermore, clearly $h(\nu(y))=h(x)=h(y)$. We deduce that $\nu\in \C[x]$.

Let $\Aut(h)$ denote the group of linear polynomials $\ell\in \C[x]$ such that $h\circ \ell=h$.
It follows that $\nu\in \Aut(h)$ and since $\nu(y)=x \neq y$, it follows that $\Aut(h)$ is a non-trivial group. Recall that $h$ is by assumption indecomposable. We now show that $\Mon(h)$ is cyclic, and hence that $h$ is cyclic. This has been shown in Remark 2.14 in \cite{ZM}, as well as in Corollary 6.6 in \cite{KZ14}. For the sake of completeness we recall the proof.

First recall from Section~\ref{sec2} that if $L$ is the splitting field of  $h(X)-t$ over $\C(t)$ and $x$ is such that $h(x)=t$, then $G:=\Mon(h)=\Gal(L/\C(h(x)))$, and if we set $H=\Gal(L/\C(x))$, then $G=HI$ for some cyclic group $I$. Now note that $\Aut(h)\cong N_G(H)/H$. Since $h$ is indecomposable, there are no intermediate fields between $\C(h(x))$ and $\C(x)$, and thus no proper subgroups between $H$ and $G$, so either $N_G(H) = G$ or $N_G(H) = H$. In the latter case, $\Aut(h)$ is trivial, a contradiction. Thus $H\unlhd G$. Since $H$ contains no nontrivial normal
subgroups of $G$ (because $L$ is the normal closure of $\C(x)/\C(h(x))$), we must
have $H = 1$, and $G=HI=I$, so $G$ is cyclic. By Proposition~\ref{cydih} it follows that $h$ is cyclic.
\end{proof}

\begin{Le}\label{ours}
Let $h\in \C[x]$ be indecomposable and let  $y\neq x$ be a root of \mbox{$h(X)-h(x)$}$\in\C(x)[X]$. Then the following hold.
\begin{enumerate}
\item For $q\in \C[h(x)]$ we have $q(x)=q(y)$. Furthermore, if $h$ is not cyclic and $q(x)=q(y)$ for some $q\in \C[x]$, then $q\in \C[h(x)]$.
\item Let $d:=\left[\C(x,y):\C(x)\right]$. Then $d\le\deg h-1$.
\item The genus of the function field $\C(x,y)$ \textup{(}over $\C$\textup{)} is not greater than $(d-1)(d-2)/2$.
\end{enumerate}
\end{Le}

Zannier~\cite[Lemma~3]{Z07} showed that for an arbitrary $h\in \C[x]$ with $\deg h\ge1$, there exists a conjugate $y$ of $x$ over $\C(h(x))$ with the above  properties:  (1) then states that for $q\in \C[x]$, we have $q(x)=q(y)$ if and only if $q\in \C[h(x)]$, while (2) and (3) are the same as above. Note that in Lemma~\ref{ours}, we put some conditions on $h$, but $y$ is an arbitrary conjugate of $x$ (such that $y\neq x$).

\begin{proof}[Proof of Lemma~\ref{ours}] The first statement follows from $h(x)=h(y)$.
Assume now that $h$ is not cyclic and that $q(x)=q(y)$ for some $q\in \C[x]$. By Lemma~\ref{either} it follows that  $\C(x)\cap\C(y)=\C(h(x))$. Since $q(x)=q(y)$, it follows that $q(x)\in \C(x)\cap\C(y)=\C(h(x))$. Furthermore, since $h, q\in \C[x]$, we have $q(x)\in \C[h(x)]$. This completes the proof of (1). We prove the other two statements completely analogously to the proof of Lemma~3 from \cite{Z07}. By setting $H(X, Y):=(h(X)-h(Y)/(X-Y)$ we have $H(x, y)=0$. Then (2) follows from $\deg_Y H\leq \deg H=\deg h-1$. If $H_1(X,Y)\in \C[X,Y]$ is any irreducible polynomial such that $H_1(x,y)=0$, then one shows by the same argument as in the proof of Lemma~\ref{extdeg} that $\deg H_1=\deg_X H_1=\deg_YH_1=d$. Then (3) is a consequence of the fact that the genus of a plane curve of degree $\leq d$ is bounded by $(d-1)(d-2)/2$.
\end{proof}


\section{Proof of Theorem \ref{main}}\label{sec5}

In this section we prove Theorem \ref{main} using results from the previous two sections. Recall that $A_0, A_1, G_0, G_1\in \C[x]$ and $(G_n)_{n=0}^\infty$ is a sequence of polynomials defined by the minimal non-degenerate simple linear recurrence
$$G_{n+2}(x)=A_{1}(x)G_{n+1}(x)+A_0(x)G_n(x), \quad n\in\N.$$
We are assuming that for some $n$ we have $G_n=g\circ h$, where $h$ is indecomposable, and that $x$ and $y$, which define equation~\ref{sum}, are such that $h(x)=h(y)$ and $x\neq y$. We will use this notation throughout this section. In this notation, we have the following characterization of the existence of a proper vanishing subsum of \eqref{sum} in the case when $\C(x)\cap \C(y)=\C(h(x))$. Note that by Lemma~\ref{either}, either this holds or $h$ is cyclic.

\begin{Le}\label{lemma}
If $\C(x)\cap \C(y)=\C(h(x))$, then there exists a proper vanishing subsum of \eqref{sum} if and only if $\pi_1\pi_2A_0(x)^n\in \C(h(x))$.
\end{Le}

Note that we have $\alpha_1+\alpha_2=A_1(x)$ and $\alpha_1\alpha_2=-A_0(x)$ by Vieta's formulae. Clearly, $G_0(x)=\pi_1+\pi_2$ and $G_1(x)=\pi_1\alpha_1+\pi_2\alpha_2$. Then
\begin{equation}\label{pis}
\pi_1=\frac{G_1(x)-\alpha_2G_0(x)}{\alpha_1-\alpha_2}, \quad \pi_2=-\frac{G_1(x)-\alpha_1G_0(x)}{\alpha_1-\alpha_2},
\end{equation}
and hence
\begin{equation}\label{productpi}
\pi_1\pi_2=-\frac{G_1(x)^2-G_0(x)G_1(x)A_1(x)-A_0(x)G_0(x)^2}{A_1(x)^2+4A_0(x)}\in \C(x).
\end{equation}
Analogously,
\begin{equation}\label{productrho}
\rho_1\rho_2=-\frac{G_1(y)^2-G_0(y)G_1(y)A_1(y)-A_0(y)G_0(y)^2}{A_1(y)^2+4A_0(y)}\in \C(y).
\end{equation}

\begin{proof}[Proof of Lemma~\ref{lemma}]
There exists a proper vanishing subsum of \eqref{sum} if and only if there exists a permutation $\sigma$ of the set $\{1, 2\}$ such that
\begin{equation}\label{perm}
\pi_i\alpha_i^n=\rho_{\sigma(i)}\beta_{\sigma(i)}^n
\end{equation}
 for $i=1, 2$. If there exists such a permutation, then in particular we have $\pi_1\pi_2 A_0(x)^n=\rho_1\rho_2A_0(y)^n$,
by Vieta's formulae.
Since $\pi_1\pi_2\in \C(x)$ and $A_0(x)\in \C(x)$ we have that $\pi_1\pi_2A_0(x)^n\in \C(x)$. Analogously $\rho_1\rho_2A_0(y)^n\in \C(y)$, so
$\pi_1\pi_2A_0(x)^n\in \C(x)\cap \C(y)=\C(h(x))$.

Assume now that $\pi_1\pi_2A_0(x)^n\in \C(h(x))$, so that $\pi_1\pi_2A_0(x)^n=p(h(x))$ for some $p\in \C(x)$. Then analogously $\rho_1\rho_2A_0(y)^n=p(h(y))$ and since $h(x)=h(y)$ we get $\pi_1\pi_2A_0(x)^n=\rho_1\rho_2A_0(y)^n$.
Since $G_n(x)=G_n(y)$ it follows that
\[
G_n(x)^2-4\pi_1\pi_2(-A_0(x))^n=G_n(y)^2-4\rho_1\rho_2(-A_0(y))^n,
\]
and hence
\[
\pi_1\alpha_1^n-\pi_2\alpha_2^n=\pm(\rho_1\beta_1^n-\rho_2\beta_2^n).
\]
Thus, there exists a proper vanishing subsum of \eqref{sum}.
\end{proof}

Note that by Lemma~\ref{either} and Lemma~\ref{lemma} it follows that if $A_0(x)=a_0\in \C$ and
\[
\pi_1\pi_2=-\frac{G_1(x)^2-G_0(x)G_1(x)A_1(x)-A_0(x)G_0(x)^2}{A_1(x)^2+4A_0(x)}=\pi\in \C,
\]
then either $h$ is cyclic or there exists a proper vanishing subsum of \eqref{sum}. On the other hand, we have the following.

\begin{Le}\label{lem0}
If $\pi_1\pi_2=\pi A_0(x)^m$ for some $m\geq 0$, $\pi\in \C$ and $\deg A_0=1$, then either $h$ is cyclic or there does not exist a proper vanishing subsum of \eqref{sum}.
\end{Le}
\begin{proof}
By $\pi_1\pi_2=\pi A_0(x)^m$ and by Lemma~\ref{either} and Lemma~\ref{lemma}, it follows that if there exists a proper vanishing subsum of \eqref{sum}, then either $h$ is cyclic or $A_0(x)^{m+n}\in \C[h(x)]$. Assuming the latter, by Lemma~\ref{ours} we have $A_0(x)=\zeta A_0(y)$ for some $(m+n)$-th root of unity $\zeta$. Then $A_0(x)\in \C(x)\cap \C(y)=\C(h(x))$. Since $\deg A_0=1$ and $\deg h\geq 2$, we have a contradiction.
\end{proof}

In Theorem~\ref{main} we are assuming that we do not have $G_m(x)\in\C[h(x)]$ for all $m\in\N$. We have the following characterization of this situation.

\begin{Le}\label{newl}
We have that $G_m(x)\in\C[h(x)]$ for all $m\in\N$  if and only if $G_0,G_1,A_0,A_1\in \C[h(x)]$.
\end{Le}
\begin{proof}
Note that if $G_0,G_1,A_0,A_1\in \C[h(x)]$ for some polynomial $h\in\C[x]$, then by the recurrence relation it follows that $G_m(x)\in\C[h(x)]$ for every $m\in\N$.

Conversely, assume that $G_m(x)\in\C[h(x)]$ for all $m\in\N$. If $G_0,G_1,G_2,G_3$ (or any four consecutive elements of the sequence) satisfy $G_1^2-G_0G_2\neq 0$, then the linear system $G_2=A_1G_1+A_0G_0,G_3=A_1G_2+A_0G_1$ shows that
\[
A_0=\frac{G_1G_3-G_2^2}{G_1^2-G_0G_2},\quad A_1=\frac{G_1G_2-G_0G_3}{G_1^2-G_0G_2}
\]
and hence $A_0(x),A_1(x)$ are in $\C(h(x))\cap\C[x]=\C[h(x)]$ (the last equality follows immediately by integrality).
Since $G_m(x)\in\C[h(x)]$ for all $m\in\N$ it cannot always hold that $G_{m+1}^2=G_mG_{m+2}$ because in this case a short calculation shows that
\[
G_{m+1}=\left(A_1\pm\sqrt{A_1^2+4A_0}\right)G_m/2,
\]
contradicting the assumption that $(G_n)_{n=0}^\infty$ is a second order linear recurrence (observe that in this case necessarily $\sqrt{A_1(x)^2+4A_0(x)}\in\C[x]$).
\end{proof}

\begin{Le}\label{third}If $h$ is not cyclic and if $G_m(x)\in\C[h(x)]$ for all $m\in\N$, then  \eqref{sum} has a proper vanishing subsum.
\end{Le}
\begin{proof}
Since $h$ is not cyclic, by Lemma \ref{either} it follows that $C(x)\cap\C(y)=\C(h(x))$. Assume that $G_m(x)\in\C[h(x)]$ for all $m\in\N$. Then by Lemma~\ref{newl} it follows that $G_0(x), G_1(x), A_0(x), A_1(x)\in\C[h(x)]$. From \eqref{productpi} we conclude that $\pi_1\pi_2\in\C(h(x))$ and hence $\pi_1\pi_2A_0(x)^n\in\C(h(x))$. By Lemma \ref{lemma} it follows that  \eqref{sum} has a proper vanishing subsum.
\end{proof}

We complete a proof of Theorem~\ref{main} with the help of two lemmas. First note that by \eqref{pis} it follows that $\pi_1=\pi_2$ if and only if $2G_1(x)=G_0(x)A_1(x)$.

\begin{Le}\label{lem1}
If $h$ is neither dihedral nor cyclic, and it does not hold that $G_m(x)\in\C[h(x)]$ for all $m$,
then \eqref{sum} has no proper vanishing subsum if $A_0(x)$ is constant  and  $2G_1(x)=G_0(x)A_1(x)$, i.e.\@ $\pi_1=\pi_2$.

Furthermore, if $h$ is not cyclic, and it does not hold that $G_m(x)\in\C[h(x)]$ for all $m$, then \eqref{sum} has no proper vanishing subsum if $\pi_1=\pi_2=\pi\in \C$ and either $\deg A_0=1$ or $\sqrt{A_1(x)^2+4A_0(x)}\in \C[x]$.
\end{Le}

\begin{proof}
Assume that $h$ is not cyclic, and it does not hold that $G_m(x)\in\C[h(x)]$ for all $m$, and that there exists a proper vanishing subsum of \eqref{sum}. Recall that by Lemma~\ref{either} and Lemma~\ref{lemma} it follows that $\pi_1\pi_2A_0(x)^n\in \C(h(x))$.

Assume first that $A_0(x)=a_0\in \C$  and $\pi_1=\pi_2=:\pi$. Then $G_0(x)=2\pi$ and since $A_0(x)=a_0\in \C$, it follows that
\[
\pi_1\pi_2A_0(x)^n=\frac{a_0^nG_0(x)^2}{4}\in \C(h(x)),
\]
and hence $G_0(x)^2\in \C(h(x))$. Then $G_0(x)^2=G_0(y)^2$ by Lemma~\ref{ours}, so $G_0(x)= \pm G_0(y)$. Thus, $G_0(x)\in \C(x)\cap \C(y)=\C(h(x))$. Moreover, $G_0(x)\in \C(h(x))\cap\C[x]=\C[h(x)]$.

Furthermore, by Proposition~\ref{propd} we have
\[
 \quad G_n(x)=\pi (\alpha_1^n+\alpha_2^n) =\frac{1}{2}G_0(x)D_n(A_1(x), -a_0)\in \C[h(x)].
\]
Since $G_0(x)\in \C[h(x)]$,
it follows that $D_n(A_1(x), -a_0)\in \C[h(x)]$.
Observe that $\deg A_1>1$. Namely, if $A_1(x)=a_1\in\C$, since $G_0(x)\in \C[h(x)]$ it follows that for any $m\in\N$ we have that
\[
G_m(x)=\frac{1}{2}G_0(x)D_m(a_1,-a_0)\in\C[h(x)],
\]
a contradiction with the assumption. If $\deg A_1=1$, we have
\[G_n(x)=\frac{1}{2}G_0(x)D_n(A_1(x),-a_0).\]
Since $G_n(x),\ G_0(x)\in\C[h(x)]$, it follows that
\[
D_n(A_1(x),-a_0)\in\C(h(x))\cap\C[x]=\C[h(x)].
\]
Obviously, $D_n(A_1(x),-a_0)$ is equivalent to $D_n(x,-a_0)$, which is either cyclic or dihedral. By Lemma \ref{sequen}, Proposition \ref{invar} and Proposition \ref{cydih} it follows that $h$ is either cyclic or dihedral, a contradiction with the assumption.
We conclude that $\deg A_1,\deg h>1$ and $A_1$ and $h$ have a common composite.
We now use Proposition \ref{indcor}. If $A_1(x)\in \C[h(x)]$, since $G_0(x)\in \C[h(x)]$ it follows that for any $m\in \N$ we have
\[
G_m(x)=\frac{1}{2}G_0(x)D_m(A_1(x), -a_0)\in \C[h(x)],
\]
a contradiction with the assumption.
Assume thus that $A_1(x)\notin \C[h(x)]$. By Proposition \ref{indcor}, since $h$ is neither cyclic nor dihedral it follows that
\[
h(x)=\ell_2(x)\circ x^rP(x^s)\circ \ell_3(x), \quad A_1(x)=\ell_1(x)\circ x^s\circ \ell_3(x)
\]
for some linear polynomials $\ell_1, \ell_2, \ell_3\in \C[x]$ and $s, r\in \N$, $s\ge2$. In particular, $A_1$ is cyclic. By Proposition~\ref{invar} and Lemma~\ref{sequen}  it follows that the collection of  monodromy groups in any complete decomposition of  $D_n(A_1(x), -a_0)$ consists only of cyclic or dihedral groups. Since $D_n(A_1(x), -a_0)\in \C[h(x)]$, by Proposition~\ref{cydih} it follows that $h$ is either cyclic or dihedral, a contradiction.

We now prove the second statement. Assume that $\pi_1=\pi_2=\pi\in \C$ and either $\deg A_0=1$ or $\sqrt{A_1(x)^2+4A_0(x)}\in \C[x]$. Then
\[
G_0(x)=2\pi, \quad G_{m}(x)=\pi D_{m}(A_1(x), -A_0(x))\quad\text{for }m\in\N.
\]
Recall that by Lemma~\ref{either} and Lemma~\ref{lemma} it again follows that
\[
\pi_1\pi_2A_0(x)^n=\pi^2A_0(x)^n\in \C(h(x))\cap \C[x]=\C[h(x)].
\]
It follows that $A_0(x)^n\in \C[h(x)]$ and thus $A_0(x)^n=A_0(y)^n$ by Lemma~\ref{ours}. Then $A_0(x)=\zeta A_0(y)$ for some $n$-th root of unity $\zeta$, so $A_0(x)\in \C(x)\cap \C(y)=\C(h(x))$, and moreover $A_0(x)\in \C[h(x)]$. In particular, $A_0(x)=A_0(y)$.  If $\deg A_0=1$ we have a contradiction since $\deg h\geq 2$.

Thus
$\sqrt{A_1(x)^2+4A_0(x)}\in \C[x]$.
Since
\[
 D_n(A_1(x), -A_0(x))=\frac{1}{\pi}G_n(x)\in \C[h(x)],
\]
 by Lemma~\ref{ours} we have
\[
D_n(A_1(x), -A_0(x))=D_n(A_1(y), -A_0(y)).
\]
Since $A_0(x)=A_0(y)$ we further get
\[
D_n(A_1(x), -A_0(x))=D_n(A_1(y), -A_0(x)).
\]
Using Proposition \ref{propd} we get that either $A_1(x)=\pm A_1(y)$  or
\begin{equation}\label{gamdel}
A_1(x)^2-\gamma_kA_1(x)A_1(y)+A_1(y)^2-\delta_k^2A_0(x)=0,
\end{equation}
for $\gamma_k,\delta_k\in \C$ given in the proposition.
If $A_1(x)=\pm A_1(y)$, then we have $A_1(x)\in \C(x)\cap \C(y)=\C(h(x))$. Then clearly $A_1(x)\in \C[(h(x)]$. Since also $A_0(x)\in \C[h(x)]$ we have that for any $m$
\[
G_m(x)=\pi D_{m}(A_1(x), -A_0(x))\in \C[h(x)],
\]
a contradiction with the assumption.
Thus we get that \eqref{gamdel} holds. A short calculation shows that
\[
A_1(x)=\frac{\gamma_kA_1(y)\pm \delta_k\sqrt{A_1(y)^2+4A_0(y)}}{2}.
\]
Since $\sqrt{A_1(x)^2+4A_0(x)}\in \C[x]$, we have that $A_1(x)\in \C(x)\cap \C(y)=\C(h(x))$. Moreover, $A_1(x)\in \C[h(x)]$. Then $A_1(x)=A_1(y)$, a contradiction with the assumption
\end{proof}

\begin{Le}\label{lem2}
If $h$ is neither dihedral nor cyclic, and it does not hold that $G_m(x)\in\C[h(x)]$ for all $m$, then \eqref{sum} has no proper vanishing subsum if $A_0(x)$ is constant, $G_0(x)=A_1(x)$, and either $G_1(x)=2A_0(x)+G_0(x)^2$ or $G_1(x)=-2A_0(x)$.

Furthermore, if $h$ is not cyclic, and it does not hold that $G_m(x)\in\C[h(x)]$ for all $m$, then \eqref{sum} has no proper vanishing subsum if  $G_0(x)=A_1(x)$ and any of the following holds:
\begin{enumerate}
\item[i)]     $G_1(x)=2A_0(x)+G_0(x)^2$,  and \\ either $\sqrt{A_1(x)^2+4A_0(x)}\in \C[x]$ or $\deg A_0=1$,
\item[ii)]    $G_1(x)=-2A_0(x)$, and  \\ either $\sqrt{A_1(x)^2+4A_0(x)}\in\C[x]$ or $\deg A_0=1$.
\end{enumerate}
\end{Le}

\begin{proof}
Assume that $h$ is not cyclic, and it does not hold that $G_m(x)\in\C[h(x)]$ for all $m$, and that there exists a proper vanishing subsum of \eqref{sum}. Recall that by Lemma~\ref{either} and Lemma~\ref{lemma} it follows that $\pi_1\pi_2A_0(x)^n\in \C(h(x))$.

Assume further  that $G_0(x)=A_1(x)$, and either $G_1(x)=2A_0(x)+G_0(x)^2$ or $G_1(x)=-2A_0(x)$. Then
\begin{align*}
\pi_1+\pi_2&=G_0(x)=A_1(x)=\alpha_1+\alpha_2\\
\pi_1\pi_2&=-\frac{G_1(x)^2-G_0(x)G_1(x)A_1(x)-A_0(x)G_0(x)^2}{A_1(x)^2+4A_0(x)}=-A_0(x)=\alpha_1\alpha_2.
\end{align*}
Thus, either $\pi_1=\alpha_1$ and  $\pi_2=\alpha_2$, or  $\pi_1=\alpha_2$ and  $\pi_2=\alpha_1$. In both cases,
\[
\pi_1\pi_2A_0(x)^n=-A_0(x)^{n+1}\in \C(h(x)).
\]
Then $A_0(x)^{n+1}\in \C(h(x))\cap \C(x)=\C[h(x)]$.
By Lemma~\ref{ours} it follows that $A_0(x)^{n+1}=A_0(y)^{n+1}$. Then $A_0(x)=\zeta A_0(y)$ for some $(n+1)$-st root of unity $\zeta$, so $A_0(x)\in \C(x)\cap \C(y)=\C(h(x))$, and moreover $A_0(x)\in \C[h(x)]$. In particular, $A_0(x)=A_0(y)$.

Since either $\pi_1=\alpha_1$ and  $\pi_2=\alpha_2$, or  $\pi_1=\alpha_2$ and  $\pi_2=\alpha_1$, by Proposition~\ref{propd} we have
\begin{equation}\label{G_m}
G_m(x)=
  \begin{cases}
    D_{m+1}(A_1(x), -A_0(x)), & \text{if}\  \pi_1=\alpha_1, \ \pi_2=\alpha_2\\
    -A_0(x)D_{m-1}(A_1(x), -A_0(x)), & \text{if}\  \pi_1=\alpha_2, \ \pi_2=\alpha_1.\\
  \end{cases}
\end{equation}
for any $m$. Since $A_0(x)\in \C[h(x)]$ and $G_n(x)\in \C[h(x)]$ it follows that for some $i\in \{n-1, n+1\}$ we have
 \begin{equation}\label{concl}
D_{i}(A_1(x), -A_0(x))\in \C[h(x)],
\end{equation}
and consequently by Lemma~\ref{ours} that
\[
D_{i}(A_1(x), -A_0(x))=D_{i}(A_1(y), -A_0(x)).
\]
Then either $A_1(x)=\pm A_1(y)$ or
\[
A_1(x)^2-\gamma_kA_1(x)A_1(y)+A_1(y)^2-\delta_k^2A_0(x)=0,
\]
for some of $\gamma_k,\delta_k$ given in Proposition \ref{propd}.
If $A_1(x)=\pm A_1(y)$, then $A_1(x)\in \C(x)\cap \C(y)=\C(h(x))$. Then clearly $A_1(x)\in \C[(h(x)]$. Since also $A_0(x)\in \C[h(x)]$, by \eqref{G_m} we have that $G_m(x)\in \C[h(x)]$ for any $m$, a contradiction with the assumption. Thus we get  that \eqref{gamdel} holds. A short calculation shows that
\[
A_1(x)=\frac{\gamma_kA_1(y)\pm \delta_k\sqrt{A_1(y)^2+4A_0(y)}}{2}.
\]
If  $\sqrt{A_1(x)^2+4A_0(x)}\in \C[x]$, we have that $A_1(x)\in \C(x)\cap \C(y)=\C(h(x))$. Moreover, $A_1(x)\in \C[h(x)]$. Then $A_1(x)=A_1(y)$, a contradiction with the assumption.
If $\deg A_0=1$ we have a contradiction with $A_0(x)\in \C[h(x)]$ since $\deg h\geq 2$.

 It remains to examine the case when in addition to the assumptions stated at the beginning of the proof we have $A_0(x)=a_0\in \C$ and $h$ is not dihedral.

By \eqref{concl} we have $D_{i}(A_1(x), -a_0)\in \C[h(x)]$.  If $A_1(x)\in \C[h(x)]$,  by \eqref{G_m} it follows that  $G_m(x)\in \C[h(x)]$ for any $m\in \N$,
a contradiction with the assumption.
Assume thus that $A_1(x)\notin \C[h(x)]$.
As in the proof of Lemma \ref{lem1}, we conclude $\deg A_1>1$.
By Proposition \ref{indcor}, since $h$ is neither cyclic nor dihedral it follows that
\[
h(x)=\ell_2(x)\circ x^rP(x^s)\circ \ell_3(x), \quad A_1(x)=\ell_1(x)\circ x^s\circ \ell_3(x)
\]
for some linear polynomials $\ell_1, \ell_2, \ell_3\in \C[x]$ and $s, r\in \N$, $s\ge2$. In particular, $A_1$ is cyclic. By Proposition~\ref{invar} and Lemma~\ref{sequen}  it follows that the collection of  monodromy groups in any complete decomposition of  $D_i(A_1(x), -a_0)$ consists only of cyclic or dihedral groups. Since $D_i(A_1(x), -a_0)\in \C[h(x)]$, by Proposition~\ref{cydih} it follows that $h$ is either cyclic or dihedral, a contradiction.

\end{proof}

\begin{proof}[Proof of Theorem~\ref{main}]
By Lemma~\ref{lem0}, Lemma~\ref{lem1} and  Lemma~\ref{lem2} we conclude the proof of Theorem~\ref{main}.
\end{proof}

\section{Proof of Theorem \ref{thm1}}\label{sec4}

\begin{proof}[Proof of Theorem \ref{thm1}]
Assume that $G_n(x)=g(h(x))$, where $h$ is indecomposable and neither cyclic nor dihedral. Recall that $x$ and $y$, which define \eqref{sum}, are  such that $h(x)=h(y)$ and $x\neq y$. From Lemma~\ref{either} it follows that $\C(x)\cap \C(y)=\C(h(x))$.  Assume further that there is no proper vanishing subsum of \eqref{sum} and write it as
\begin{align}
\label{equ2}
1-\frac{\pi_1\alpha_1^n}{\rho_2\beta_2^n}+\frac{\rho_1\beta_1^n}{\rho_2\beta_2^n}-\frac{\pi_2\alpha_2^n}{\rho_2\beta_2^n}=0.
\end{align}
Define
\[
u_1=-\frac{\pi_1\alpha_1^n}{\rho_2\beta_2^n}, \quad u_2=\frac{\rho_1\beta_1^n}{\rho_2\beta_2^n}, \quad u_3=-\frac{\pi_2\alpha_2^n}{\rho_2\beta_2^n},
\]
and also
\[
v_1=\frac{\alpha_1}{\beta_2}, \quad v_2=\frac{\beta_1}{\beta_2}, \quad v_3=\frac{\alpha_2}{\beta_2},
\]
\[
w_1=\frac{\pi_1}{\rho_2}, \quad w_2=\frac{\rho_1}{\rho_2}, \quad w_3=\frac{\pi_2}{\rho_2}.
\]
Let $F=\C(x,y,\alpha_1,\alpha_2, \beta_1, \beta_2)$ and let $\HH$ be the projective height on $F/\C(x)$, defined as in Section~\ref{sec3}.
By Lemma \ref{propert}, we find the estimate
\begin{align*}
\HH\left(\frac{\pi_i\alpha_i^n}{\rho_2\beta_2^n}\right)\ge n\cdot\HH\left(\frac{\alpha_i}{\beta_2}\right)-\HH\left(\frac{\pi_i}{\rho_2}\right),\quad i=1,2,
\end{align*}
and similarly we argue for $u_2$.
So, for $i=1,2,3$, we have
\[
\HH(u_i)\ge n\HH(v_i)-\HH(w_i).
\]
Note that if for some $i$ we have $\HH\left(v_i\right)\neq0$, then
\[
n\le\left(\HH(u_i)+\HH(w_i)\right)\cdot\HH(v_i)^{-1}.
\]
Since  $(G_n(x))_{n=0}^\infty$ is non-degenerate, the same holds for the sequence $(G_n(y))_{n=0}^\infty$, i.e. $\beta_1/\beta_2\notin \C$. It follows that $\HH(v_2)=\HH(\beta_1/\beta_2)\neq 0$ and thus
\begin{align}\label{Hu_2}
n \leq (\HH(u_2)+\HH(w_2))\cdot \HH(v_2)^{-1}.
\end{align}
On the other hand, we find the following upper bound for the height of $G_n(x)$:
\begin{align*}
\HH(G_n(x))&=\HH(\pi_1\alpha_1^n+\pi_2\alpha_2^n)\\
				&\le \HH(\pi_1)+n\HH(\alpha_1)+\HH(\pi_2)+n\HH(\alpha_2)\\
				&\le n(\HH(\alpha_1)+\HH(\alpha_2)+\HH(\pi_1)+\HH(\pi_2)).
\end{align*}
Using \eqref{Hu_2}, we conclude that
\begin{align}\label{equ3}
\HH(G_n(x)) \le\left(\HH(u_2)+\HH(w_2)\right)\HH(v_2)^{-1}(\HH(\alpha_1)+\HH(\alpha_2)+\HH(\pi_1)+\HH(\pi_2)).
\end{align}

Now consider equation (\ref{equ2}), which by assumption has no proper vanishing subsum.
Let $A=\{\alpha_i,\pi_i,\beta_i,\rho_i,i=1,2\}$ and put
\[
S:=\{\nu\in S_0:\ \nu(f)\neq0 \text{ for some } f\in A\}\cup S_\infty,
\]
where  $S_0$ denotes the set of finite valuations and $S_\infty$ denotes the set of infinite valuations on $F$.
Then by Theorem \ref{bm} it follows that
\begin{align}\label{bm2}
\HH(u_2)\le|S|+2\g-2,
\end{align}
where $\g$ is the genus of $F/\C$. We now estimate the genus and $|S|$ in terms of $\deg h$.
We start with the genus. In order to use Castelnuovo's inequality (Theorem \ref{castelnuovo}), we define
\begin{align*}
F_1=\C(x,\alpha_1,\alpha_2), \quad F_2=\C(y,\beta_1,\beta_2).
\end{align*}
Note that $\C$ is the field of constants of $F_1,F_2$ and that $F=F_1F_2$. Let $n_i:=\left[F:F_i\right]$, $i=1,2$. Recall that the $\alpha_i$'s and $\beta_i$'s are  roots of a monic quadratic polynomial and that $[\C(x, y):\C(x)]< \deg h$ by Lemma~\ref{ours}. Thus $n_i< 2\deg h$. For $i=1,2$ let $\g_i$ be the genus of $F_i/\C$.
Note that since
\begin{equation}\label{alpha}
\alpha_1=\frac{A_1(x)-\sqrt{A_1(x)^2+4A_0(x)}}{2}, \quad\alpha_2=\frac{A_1(x)+\sqrt{A_1(x)^2+4A_0(x)}}{2},
\end{equation}
we have that $\C(x,\alpha_1,\alpha_2)=\C(x,\sqrt{\Delta(x)})$, where $\Delta(x)=A_1(x)^2+4A_0(x)$. Now Riemann's inequality (Theorem \ref{riemann}) yields
\begin{align*}
\g_1\le([F_1:\C(x)]-1)([F_1:\C(\sqrt{\Delta(x)})]-1).
\end{align*}
Since $\sqrt{\Delta(x)}$ is a root of $T^2-\Delta(x)\in\C(x)[T]$ and $x$ is a root of $\Delta(T)-\sqrt{\Delta(x)}^2\in\C(\sqrt{\Delta(x)})[T]$, we conclude that
\[
\g_1\le(2-1)\cdot(\deg \Delta -1)=\deg \Delta -1 \leq C_1-1,
\]
where $C_1:=\max\{\deg A_0,2\deg A_1\}$ (it will be shown later that indeed $C_1\ge 1$, by the non-degeneracy of the sequence).

Since $F_1$ and $F_2$ are isomorphic function fields, they have the same genus and hence the same bound holds for $\g_2$. Therefore we find that $\g_i\le C_1-1,\ i=1,2$.
By Castelnuovo's inequality (Theorem \ref{castelnuovo}) we get
\begin{align*}
\g&\le n_1\g_1+n_2\g_2+(n_1-1)(n_2-1)\\
&<4\deg h  (C_1-1)+(2\deg h-1)^2<4C_1 \deg h^2.
\end{align*}
To estimate $|S|$, let
\begin{align*}
S_1&=\left\{\nu\in S_0:\ \nu(\alpha_1)\neq0\text{ or }\nu(\alpha_2)\neq0\right\},\\
S_2&=\left\{\nu\in S_0:\ \nu(\pi_1)\neq0\text{ or }\nu(\pi_2)\neq0\right\},\\
S_3&=\left\{\nu\in S_0:\ \nu(\beta_1)\neq0\text{ or } \nu(\beta_2)\neq0\right\},\\
S_4&=\left\{\nu\in S_0:\ \nu(\rho_1)\neq0\text{ or } \nu(\rho_2)\neq0\right\}.
\end{align*}
Clearly, $|S|\le|S_1|+|S_2|+|S_3|+|S_4|+|S_{\infty}|$. Since $[F:\C(x)]< 4\deg h$ we have $|S_\infty|<4\deg h$. For the other sets, we argue as follows.

Note that the $\alpha_i$'s are integral over $\C(x)$ (they are the roots of $\mathcal{G}(T)$) and therefore $\nu(\alpha_i)\ge0$ for every finite valuation $\nu$. Note that thus $\nu(\alpha_1\alpha_2)>0$ if and only if either $\nu(\alpha_1)>0 \text{ or } \nu(\alpha_2)>0$. Also, by Vieta's formulae we have $\alpha_1\alpha_2=-A_0(x)$. Further recall that by Lemma~\ref{propert} we have $\HH(A_0(x))=\deg A_0\cdot\HH(x)$ and that $\sum_{\nu}\max(0,\nu(A_0(x)))=\HH(A_0(x))$ by the sum formula. Thus,
\begin{align*}
|S_1|&=|\left\{\nu\in S_0:\ \nu(\alpha_1)>0 \text{ or } \nu(\alpha_2)>0\right\}|\\
&=|\left\{\nu\in S_0:\ \nu(\alpha_1\alpha_2)>0\right\}|=|\{\nu\in S_0:\ \nu(A_0(x))>0\}|\\
&\textstyle\le\sum'1\le \sum\limits_{\nu}\max(0,\nu(A_0(x)))=\HH(A_0(x))\\
&=\deg A_0\cdot\HH(x)=\deg A_0\cdot [F:\C(x)]< \deg A_0\cdot4\deg h,
\end{align*}
where the sum $\sum'$ runs over all valuations $\nu$ for which $\nu(A_0(x))>0$ holds.

In order to bound $|S_3|$ we argue similarly. We have that $\beta_1,\beta_2$ are the roots of the characteristic polynomial of $(G_n(y))_{n=0}^\infty$, and are hence integral over $\C(y)$. Since $y$ is integral over $\C(x)$, we have that $\beta_1, \beta_2$ are integral  over $\C(x)$. Therefore, as in the the case of $S_1$ we conclude that $\nu(\beta_i)\ge0$ for every finite valuation $\nu$. By Vieta's formulae we have $\beta_1\beta_2=-A_0(y)$. Furthermore, since $h(x)=h(y)$ we have
\[
\deg h\cdot\HH(y)=\HH(h(y))=\HH(h(x))=\deg h\cdot\HH(x),
\]
and thus
\[
\HH(y)=\HH(x)=[F:\C(x)].
\]
Therefore,
\begin{align*}
|S_3|&=|\left\{\nu\in S_0:\ \nu(\beta_1)>0 \text{ or } \nu(\beta_2)>0\right\}|\\
&=|\{\nu\in S_0:\ \nu(\beta_1\beta_2)>0\}|=|\{\nu\in S_0:\ \nu(A_0(y))>0\}|\\
&\le\deg A_0\cdot\HH(y)=\deg A_0\cdot\HH(x)<\deg A_0\cdot4\deg h.
\end{align*}

For $|S_2|$, note that
\begin{align*}
|S_2|&\le|\{\nu\in S_0:\ \nu(\pi_1)>0\text{ or }\nu(\pi_2)>0\}|\\
&\phantom{\ \le}+|\{\nu\in S_0:\ \nu(\pi_1)<0\text{ or }\nu(\pi_2)<0\}|.
\end{align*}
Recall that  $G_0(x),G_1(x),\alpha_1,\alpha_2$ are integral over $\C(x)$, and thus also $G_1(x)-\alpha_2G_0(x)$, $G_1(x)-\alpha_1G_0(x)$ and $\alpha_1-\alpha_2$. Therefore, for any $\nu\in S_0$ we have $\nu(G_1(x)-\alpha_2G_0(x))\geq 0$, $\nu(G_1(x)-\alpha_1G_0(x))\geq 0$ and $\nu(\alpha_1-\alpha_2)\geq 0$. Thus
\begin{align*}
\nu(\pi_1)&=\nu\left(\frac{G_1(x)-\alpha_2G_0(x)}{\alpha_1-\alpha_2}\right)\\
&=\nu(G_1(x)-\alpha_2G_0(x))+\nu\left(\frac{1}{\alpha_1-\alpha_2}\right)\\
&=\underbrace{\nu(G_1(x)-\alpha_2G_0(x))}_{\ge0}-\underbrace{\nu(\alpha_1-\alpha_2)}_{\ge 0}.
\end{align*}
Hence for $\nu\in S_0$ it follows that
\begin{align*}
\nu(\pi_1)>0&\mbox{ implies }\nu(G_1(x)-\alpha_2G_0(x))>0,\\
\nu(\pi_1)<0&\mbox{ implies }\nu(\alpha_1-\alpha_2)>0.
\end{align*}
In the same manner we see that
\begin{align*}
\nu(\pi_2)>0&\mbox{ implies }\nu(G_1(x)-\alpha_1G_0(x))>0,\\
\nu(\pi_2)<0&\mbox{ implies }\nu(\alpha_1-\alpha_2)>0.
\end{align*}

Further note that since $\nu(G_1(x)-\alpha_2G_0(x))\geq 0$ and $\nu(G_1(x)-\alpha_1G_0(x))\geq 0$ for any $\nu \in S_0$ we have
that either $\nu(G_1(x)-\alpha_2G_0(x))>0$ or $\nu(G_1(x)-\alpha_1G_0(x))>0$ if and only if
\[
\nu((G_1(x)-\alpha_2G_0(x))(G_1(x)-\alpha_1G_0(x)))>0,
\]
that is
\[
\nu(G_1(x)^2-G_0(x)G_1(x)A_1(x)-A_0(x)G_0(x)^2)>0.
\]
 In a similar manner we conclude that $\nu(\alpha_1-\alpha_2)>0$ if and only if $\nu((\alpha_1-\alpha_2)^2)>0$, that is
\[
 \nu(A_1(x)^2+4A_0(x))>0.
\]
Therefore,
\begin{align*}
|S_2|\le\ &|\{\nu\in S_0:\ \nu(\pi_1)>0\text{ or }\nu(\pi_2)>0\}|\\
&+|\{\nu\in S_0:\ \nu(\pi_1)<0\text{ or }\nu(\pi_2)<0\}|\\
\leq\  & |\{\nu\in S_0:\ \nu(G_1(x)-\alpha_2G_0(x))>0\text{ or } \nu(G_1(x)-\alpha_1G_0(x))>0\}|\\
&+ |\{\nu\in S_0:\ \nu(\alpha_1-\alpha_2)>0\}|\\
=\ &|\{\nu\in S_0: \nu(G_1(x)^2-G_0(x)G_1(x)A_1(x)-A_0(x)G_0(x)^2)>0\}|\\
&+|\{\nu\in S_0:  \nu(A_1(x)^2+4A_0(x))>0\}|,\\
 \intertext{and then arguing similarly as for $S_1$ we get}
|S_2|\le& \HH((G_1^2-G_0G_1A_1-G_0^2A_0)(x))+\HH((A_1^2+4A_0)(x))\\
\le&(C_2+C_1) \HH(x)<(C_1+C_2)4\deg h ,
\end{align*}
where
\[C_2:=\max\{2\deg G_1,\deg G_0+\deg G_1+\deg A_1,2\deg G_0+\deg A_0\}.\]
We argue similarly for $|S_4|$:
\begin{align*}
|S_4|&\le|\{\nu\in S_0:\ \nu(\rho_1)>0\text{ or }\nu(\rho_2)>0\}|+|\{\nu\in S_0:\ \nu(\rho_1)<0\text{ or }\nu(\rho_2)<0\}|\\
& \leq |\{\nu\in S_0: \nu(G_1(y)^2-G_0(y)G_1(y)A_1(y)-A_0(y)G_0(y)^2)>0\}|\\
&\phantom{\ge}+|\{\nu\in S_0:  \nu(A_1(y)^2+4A_0(y))>0\}|\\
&\le \HH((G_1^2-G_0G_1A_1-G_0^2A_0)(y))+\HH((A_1^2+4A_0)(y))\\
&\le(C_1+C_2) \HH(y)=(C_1+C_2) \HH(x)<(C_1+C_2)4\deg h .
\end{align*}

This gives
\begin{align*}
|S|&\le|S_1|+|S_2|+|S_3|+|S_4|+|S_{\infty}|\\
&< 8\deg A_0\deg h+8\deg h(C_1+C_2)+4\deg h\\
&=(8(\deg A_0+C_1+C_2)+4)\deg h.
\end{align*}
Finally we get
\begin{align*}
\HH(u_2)&\le 2\g-2+|S|\notag\\
&< 8C_1\deg h^2+(8(\deg A_0+C_1+C_2)+4)\deg h - 2\notag\\
&< \left(8C_1+8(\deg A_0+C_1+C_2)+4\right)\deg h^2\notag\\
&=4(2\deg A_0+4C_1+2C_2+1)\deg h^2.\label{u}
\end{align*}
We continue to estimate the terms in ~\eqref{equ3}. To give an upper bound on $\HH(\alpha_i)$, note that for $\HH(\Delta(x))=\HH(\sqrt{\Delta(x)}^2)=2\HH(\sqrt{\Delta(x)})$ it follows that
\[
\HH(\sqrt{A_1(x)^2+4A_0(x)})=\frac{1}{2}\HH(A_1(x)^2+4A_0(x)).
\]
Therefore,  by \eqref{alpha} we get
\begin{align*}
\HH(\alpha_i)&\le \HH(A_1(x))+\HH(\sqrt{A_1(x)^2+4A_0(x)})\le \frac{3}{2}C_1\HH(x), \quad i=1, 2.
\end{align*}
Using $\HH(x)=\HH(y)$, we obtain the same upper bound for $\HH(\beta_1)$ and $\HH(\beta_2)$:
\[
 \HH(\beta_i) \le \frac{3}{2}C_1\HH(x), \quad i=1, 2.
\]
Furthermore, we have
\begin{align*}
\HH(\pi_1)+\HH(\pi_2)&=\HH\left(\frac{G_1(x)-\alpha_2 G_0(x)}{\alpha_1-\alpha_2}\right)+\HH\left(-\frac{G_1(x)-\alpha_1 G_0(x)}{\alpha_1-\alpha_2}\right)\\
&\le\HH(G_1(x))+\HH(\alpha_2)+\HH(G_0(x))+\HH(\alpha_1)+\HH(\alpha_2)+\\
&\phantom{\ge}+\HH(G_1(x))+\HH(\alpha_1)+\HH(G_0(x))+\HH(\alpha_1)+\HH(\alpha_2)\\
&=2(\HH(G_0(x))+\HH(G_1(x)))+3(\HH(\alpha_1)+\HH(\alpha_2))\\
&\leq (2(\deg G_0+\deg G_1)+9C_1)\HH(x).
\end{align*}
It therefore follows that
\begin{equation}\label{aapp}
\HH(\alpha_1)+\HH(\alpha_2)+\HH(\pi_1)+\HH(\pi_2)\le(2(\deg G_0+\deg G_1)+12C_1)\HH(x).
\end{equation}
Next, we estimate the height of $w_2$ in a similar way:
\begin{align}
\HH(w_2)=\HH\left(\frac{\rho_1}{\rho_2}\right)&=\HH\left(-\frac{G_1(y)-\beta_2 G_0(y)}{G_1(y)-\beta_1 G_0(y)}\right)\notag \\
&\le2(\HH(G_1(y))+\HH( G_0(y)))+\HH(\beta_1)+\HH(\beta_2)\notag \\
&\le 2(\deg G_1+\deg G_0)\HH(y)+3C_1\HH(y)\notag \\
&<(2(\deg G_0+\deg G_1)+3C_1)4\deg h\notag \\
&<(2(\deg G_0+\deg G_1)+3C_1)4\deg h^2\notag.
\end{align}
Thus
\begin{align}\label{rr}
H(u_2)+\HH(w_2)< 4\deg h^2 \left(2(\deg A_0+\deg G_0+\deg G_1)+7C_1+2C_2+1\right).
\end{align}
We now find a lower bound for $\HH(v_2)$ in terms of $\HH(x)$:
\begin{align*}
\HH(v_2)&=\HH\left(\frac{\beta_1}{\beta_2}\right)=\HH\left(\frac{A_1(y)-\sqrt{A_1(y)^2+4A_0(y)}}{A_1(y)+\sqrt{A_1(y)^2+4A_0(y)}}\right)\\
&=\HH\left(1-2\cdot\frac{\sqrt{A_1(y)^2+4A_0(y)}}{A_1(y)+\sqrt{A_1(y)^2+4A_0(y)}}\right)\\
&=\HH\left(\frac{\sqrt{A_1(y)^2+4A_0(y)}}{A_1(y)+\sqrt{A_1(y)^2+4A_0(y)}}\right)=\HH\left(\frac{A_1(y)}{\sqrt{A_1(y)^2+4A_0(y)}}+1\right)\\
&=\HH\left(\sqrt{\frac{A_1(y)^2}{A_1(y)^2+4A_0(y)}}\right)=\frac{1}{2}\HH\left(\frac{A_1(y)^2+4A_0(y)}{A_1(y)^2}\right)=\frac{1}{2}\HH\left(\frac{A_0(y)}{A_1(y)^2}\right).
\end{align*}
Note that
\begin{align*}
\HH\left(\frac{A_0(y)}{A_1(y)^2}\right)&\ge |\HH(A_0(y))-\HH(A_1(y)^2)|= |\deg A_0-2 \deg A_1|\cdot\HH(y).
\end{align*}
If $\deg A_0\neq2 \deg A_1$, then clearly
\[
\HH\left(\frac{A_0(y)}{A_1(y)^2}\right)\geq \HH(y).
\]
If on the other hand we have that $\deg A_0=2\deg A_1$, then by the polynomial remainder theorem we have that $A_0(y)=A_1(y)^2q(y)+r(y)$, where $q\in\C$ is constant and $\deg r<2\cdot\deg A_1$. Thus,
\begin{align*}
\HH\left(\frac{A_0(y)}{A_1(y)^2}\right)&=\HH\left(q(y)+\frac{r(y)}{A_1(y)^2}\right)=\HH\left(\frac{r(y)}{A_1(y)^2}\right)\\
&\ge\HH(A_1(y)^2)-\HH(r(y))=(2\deg A_1-\deg r)\cdot\HH(y)\ge \HH(y).
\end{align*}
Thus,
\begin{equation}\label{bb}
\HH(v_2))\ge\frac{1}{2}\HH(y)=\frac{1}{2}\HH(x).
\end{equation}
(Note that since $\HH(v_2)\neq0$,  we cannot have $\deg A_0=\deg A_1=0$).

Considering again \eqref{equ3}, by \eqref{aapp}, \eqref{rr} and \eqref{bb} we find that
\begin{align*}
\HH(G_n(x))< C\deg h^2,
\end{align*}
where $C=16\left(2(\deg A_0+\deg G_0+\deg G_1)+7C_1+2C_2+1\right)\left(\deg G_0+\right.$ $\left.\deg G_1+6C_1\right).$

To give a suitable lower bound for $\HH(G_n(x))$, note that since $G_n=g\circ h$ we have
\[
\HH(G_n(x))=\deg g  \deg h \cdot [F:\C(x)]=\deg g  \deg h \cdot[F:\C(x,y)]\cdot[\C(x,y):\C(x)].
\]
By Lemma \ref{extdeg} it follows that $[\C(x,y):\C(x)]\ge\frac{1}{2}\deg h$. Therefore, we have
\begin{align*}
\HH(G_n(x))\geq \frac{1}{2}\deg g\deg h^2\cdot [F:\C(x,y)]\geq \frac{1}{2}\deg g\deg h^2.
\end{align*}
Finally, we conclude that
\begin{align*}
\frac{1}{2}\deg g\deg h^2\le \HH(G_n(x))< C\deg h^2,
\end{align*}
and therefore that $\deg g< 2C$.
\end{proof}

\section{Acknowledgement} The work on this manuscript was supported by FWF (Austrian Science Fund) Grant No. P24574.

\vspace{1cm}

\end{document}